\documentclass[11pt]{article}
\usepackage[latin1]{inputenc}
\usepackage{amsmath, graphicx, natbib, verbatim}
\usepackage{setspace, mathrsfs, amssymb}
\usepackage{amsthm, cases}
\usepackage{caption,subcaption}
\usepackage{adjustbox}
\usepackage{amsmath}
\usepackage{natbib}
\usepackage{verbatim,mathrsfs,amssymb,graphicx}
\usepackage{subcaption,setspace,cases,xcolor}
\usepackage{pstricks,pst-node,pst-text,pst-3d}
\usepackage{tikz}
\usetikzlibrary{shapes.geometric, arrows}
\usepackage{hyperref}

\usepackage[hmargin=1in,vmargin=1in]{geometry}

\title{Sequential identification of nonignorable missing data mechanisms}

\author{Mauricio Sadinle and Jerome P. Reiter
\\\\
\textsc{Duke University} }
\newtheorem{theorem}{Theorem}

\newtheorem{definition}{Definition}
\newtheorem{remark}{Remark}
\newtheorem{example}{Example}

\newcommand{\bX}{X}%\newcommand{\bX}{\mathbf X}
\newcommand{\bx}{x}%\newcommand{\bx}{\mathbf x}

\newcommand{\mcX}{\mathcal X}

\newcommand{\bM}{M}%\newcommand{\bM}{\mathbf M}
\newcommand{\bm}{m}%\newcommand{\bm}{\mathbf m}

\newcommand\indep{\perp\!\!\!\perp}
\begin{document}
\maketitle

\begin{abstract}
With nonignorable missing data, 
likelihood-based inference should be based on the joint distribution of the study variables and their missingness indicators.  
These joint models cannot be estimated from the data alone, thus
requiring the analyst to impose restrictions that make the models uniquely obtainable 
from the distribution of the observed data.  We present an approach for
constructing classes of identifiable nonignorable missing data models.  
The main
idea is to use a sequence of carefully set up identifying assumptions, whereby we
specify potentially different missingness mechanisms for different blocks of variables.  
 We show that the procedure results in models with the desirable
 property of being non-parametric saturated.
\end{abstract}

\textit{Key words and phrases:} Identification; Non-parametric saturated; Missing not at random; Partial ignorability; Sensitivity analysis.

\section{Introduction}

When data are missing not at random (MNAR) (\cite{Rubin76}),  
appropriate likelihood-based inference requires explicit models for
the {\em full-data distribution}, i.e., 
the joint distribution of the study variables and their missingness indicators. 
Because of the missing data, this distribution is not 
uniquely identified from the observed data alone (\cite{LittleRubin02}).
To enable inference, analysts must impose restrictions on the full-data distribution.
Such assumptions generally are untestable; however, a minimum
desideratum is that they result in a unique full-data distribution for the {\em
  observed-data distribution} at hand, i.e., the distribution that can be
identified from the incomplete data. 

We present a strategy for constructing identifiable
full-data distributions with nonignorable missing data. In its most
general form, the strategy is to expand the observed-data distribution
sequentially by identifying parts of the full-data distribution associated with blocks
of variables, one block at a time.  This partitioning of the variables allows analysts
to specify different missingness mechanisms in the different blocks; for
example, use the missing at random (MAR, \cite{Rubin76}) assumption for some variables and a nonignorable missingness assumption
for the rest to obtain a partially ignorable mechanism (\cite{HarelSchafer09}). We ensure that the resulting full-data distributions are {\em non-parametric saturated}
(NPS, \cite{Robins97}), that is, their implied observed-data distribution  matches the actual observed-data distribution, as detailed in Section \ref{ss:NPS}.  

Related approaches to partitioning variables with missing data have
appeared previously in the literature. 
\cite{ZhouBCMAR} proposed to
model blocks of study variables and their missingness indicators in a
sequential manner; however, their approach does not guarantee
identifiability of the full-data distribution.  \cite{HarelSchafer09}
mentioned the possibility of treating the missingness in
blocks of variables differently, but they do not provide 
results on identification.  \cite{Robins97} proposed 
the group permutation missingness mechanism, which assumes MAR sequentially
for blocks of variables and results in a NPS model.  This is a
particular case of our more general procedure, as we describe in
Section \ref{ss:Robins}. 

The remainder of the article is organized as follows.  In Section
\ref{s:background}, we describe notation and provide more details on
the NPS property.  In Section \ref{s:SeqIdent}, we introduce our strategy
for identifying a full-data distribution in a sequential manner. In Section \ref{s:Examples} we present some examples of how to use this strategy for the case of two categorical study variables, for the construction of partially ignorable mechanisms, and for sensitivity analyses.  Finally, in
Section \ref{s:disc} we discuss possible future uses of our identifying
approach.
\par

\markboth{\hfill{\footnotesize\rm MAURICIO SADINLE AND JEROME P. REITER} \hfill}
{\hfill {\footnotesize\rm SEQUENTIAL IDENTIFICATION OF NONIGNORABLE MISSING DATA} \hfill}

\section{Notation and Background}\label{s:background}

\subsection{Notation}\label{notation}

Let $\bX=(X_1,\ldots,X_p)$ denote $p$ random variables taking values
on a sample space $\mcX$. Let $M_j$ be the missingness indicator for
variable $j$, where $M_j=1$ when $X_j$ is missing and $M_j=0$ when
$X_j$ is observed.  Let $\bM=(M_1,\ldots,M_p)$, which takes values on
$\{0,1\}^p$.  Let $\mu$ be a dominating measure for the
distribution of $\bX$, and let $\nu$ represent the product measure
between $\mu$ and the counting measure on $\{0,1\}^p$.  The full-data
distribution is the joint distribution of $\bX$ and $\bM$. We call its density $f$ with respect to $\nu$ the
{\em full-data density}.  In practice, the full-data distribution
cannot be recovered from sampled data, even with an infinite sample size.

An element $\bm=(m_1,\ldots,m_p)\in \{0,1\}^p$ is called a {\em missingness pattern}.  Given $\bm\in\{0,1\}^p$ we 
define $\bar{\bm}=\mathbf{1}_p-\bm$ to be the indicator vector of observed variables, where $\mathbf{1}_p$ 
is a vector of ones of length $p$.  For each $\bm$, we define $\bX_{\bm}=(X_j: m_j=1)$ to be the missing variables and $\bX_{\bar{\bm}}=(X_j: \bar m_j=1)$ to be the observed variables, which have sample spaces $\mcX_{\bm}$ and $\mcX_{\bar{\bm}}$, respectively.  
The observed-data distribution is the distribution involving the observed variables and the missingness indicators, which has density  $f(\bX_{\bar{\bm}}=\bx_{\bar{\bm}},\bM=\bm)=\int_{\mcX_{\bm}}f(\bX=\bx,\bM=\bm)\mu(d\bx_{\bm})$, where $\bx\in\mcX$ represents a generic element of the sample space, and we define $\bx_{\bm}$ and $\bx_{\bar{\bm}}$ similarly as with the random vectors.

An alternative way of representing the observed-data distribution is
by introducing the {\em materialized variables} 
$\bX^*=(X_1^*,\dots,X_p^*)$, where
\begin{equation*}
X_j^* \equiv \left\{\begin{array}{cc}
        X_j, & \text{if } M_j=0;\\
        *, & \text{if } M_j=1;
        \end{array}\right.
\end{equation*}
and ``$*$'' is a placeholder for missingness.  The sample space $\mcX_j^*$ of each  $X_j^*$ is the union of $\{*\}$ and the sample space $\mcX_j$ of $X_j$.  The materialized variables contain all the observed information: if $X_j^*=*$ then $X_j$ was not observed, and if $X_j^*=x_j$ for any value $x_j\neq *$ then $X_j$ was observed and $X_j=x_j$.  Given $\bm\in\{0,1\}^p$ and
$\bx_{\bar{\bm}}\in\mcX_{\bar{\bm}}$, we define  $\bx^*\equiv\bx^*(\bm,\bx_{\bar{\bm}})$, 
such that
$\bx_{\bar{\bm}}^*=\bx_{\bar{\bm}}$ and $\bx^*_{\bm}=\boldsymbol{*}$,
where $\boldsymbol{*}$ is a vector with the appropriate number of $*$
symbols.  For example, if $\bm=(1,1,0)$ and $\bx_{\bar{\bm}}=x_3$, then $\bx^*=(*,*,x_3)$.  The event $\bX^*=\bx^*(\bm,\bx_{\bar{\bm}})$ is equivalent to $\bM=\bm$ and $\bX_{\bar{\bm}}=\bx_{\bar{\bm}}$, which implies that 
the distribution of $\bX^*$ is equivalent to the observed-data distribution. Therefore, with some abuse of notation,  
the observed-data density can be written in terms of $\bX^*$, that is $f\{\bX^*=\bx^*(\bm,\bx_{\bar{\bm}})\}\equiv f(\bX_{\bar{\bm}}=\bx_{\bar{\bm}},\bM=\bm)$.  When there is no need to refer to the $\bm$ and $\bx_{\bar{\bm}}$ that define $\bx^*$, we simply write $f(\bX^*=\bx^*)$ to denote the observed-data density evaluated at an arbitrary point.

In what follows we often write $f(\bX=\bx,\bM=\bm)$ simply as $f(\bX,\bM)$, $f(\bX^*=\bx^*)$ as $f(\bX^*)$, and likewise for other expressions, provided that there is no ambiguity. For the sake of simplicity, we use ``$f$'' for technically different functions, but their actual interpretations should be clear from the arguments passed to them.  For example, we denote the {\em missingness mechanism} as $f(\bM=\bm| \bX=\bx)$, or simply $f(\bM| \bX)$.
\par 

\subsection{Non-Parametric Saturated Modeling}\label{ss:NPS}

Since the true joint distribution of $\bX$ and $\bM$ cannot be identified from observed data alone, 
we need to work under the assumption that the full-data distribution falls within a class defined by a set of restrictions.  

\begin{definition}[Identifiability] Consider a class of full-data distributions $\mathcal{F}_A$ defined by a set of restrictions $A$.  We say that the class $\mathcal{F}_A$ is identifiable if there is a mapping from the set of observed-data distributions to $\mathcal{F}_A$.
\end{definition}

If we only require identifiability from a set of full-data distributions, two different observed-data distributions could map to the same full-data distribution. \cite{Robins97} introduced the stricter concept of a class of full-data
 distributions being non-parametric saturated --- also called
 non-parametric identified (\cite{Vansteelandtetal06, DanielsHogan08}).

\begin{definition}[Non-parametric Saturation] Consider a class of full-data distributions $\mathcal{F}_A$ defined by a set of restrictions $A$.  We say that the class $\mathcal{F}_A$ is non-parametric saturated if there is a \emph{one-to-one} mapping from the set of observed-data distributions to $\mathcal{F}_A$.
\end{definition}

The set $A$ of restrictions, or identifiability assumptions, that define a NPS class allow us to build a full-data distribution, say with density $f_A(\bX=\bx,\bM=\bm)$, from an observed-data distribution with density $f(\bX_{\bar{\bm}}=\bx_{\bar{\bm}},\bM=\bm)$, so that
 $f_A(\bX_{\bar{\bm}}=\bx_{\bar{\bm}},\bM=\bm) = f(\bX_{\bar{\bm}}=\bx_{\bar{\bm}},\bM=\bm)$,
where by definition  
$f_A(\bX_{\bar{\bm}}=\bx_{\bar{\bm}},\bM=\bm)  = \int_{\mcX_{\bm}}f_A(\bX=\bx,\bM=\bm)\mu(d\bx_{\bm})$. 
In terms of $\bX^*$, the NPS property is expressed as $f_A(\bX^*)=f(\bX^*)$.  

NPS is a desirable property, particularly for comparing inferences
under different approaches to handling nonignorable missing data.
When two missing data models satisfy NPS, we can be sure
that any discrepancies in inferences are due entirely to the different
assumptions on the non-identifiable parts of the full-data
distribution.  In contrast, without NPS, it can be 
difficult to disentangle what parts of the discrepancies are due to
the identifying assumptions and what parts are due to differing 
constraints on the observed-data distribution. Thus, NPS greatly
facilitates sensitivity analysis (\cite{Robins97}).

For a given $\bm$, we refer to the conditional
distribution of the missing study variables given the observed data as
the {\em missing-data distribution}---also known as the extrapolation
distribution (\cite{DanielsHogan08})---with density
$f(\bX_{\bm}=\bx_{\bm}| \bX_{\bar{\bm}}=\bx_{\bar{\bm}},\bM=\bm)$.
These distributions correspond to the non-identifiable parts of the
full-data distribution. A NPS approach is equivalent to
a recipe for building these distributions from the observed-data
distribution without imposing constraints on the latter.   

NPS models can be constructed in many ways. For example, in 
pattern mixture models,  
one can use the complete-case missing-variable restriction (\cite{Little93}), which sets   
 $f(\bX_{\bm}=\bx_{\bm}|
\bX_{\bar{\bm}}=\bx_{\bar{\bm}},\bM=\bm)=f(\bX_{\bm}=\bx_{\bm}|
\bX_{\bar{\bm}}=\bx_{\bar{\bm}},\bM=\mathbf{0}_p)$, for all
$\bm\in\{0,1\}^p$.  Although \cite{Little93} considered parametric models
for each $f(\bX_{\bar{\bm}}=\bx_{\bar{\bm}}|\bM=\bm)$, this does
not have to be the case, and therefore pattern mixture models can be
NPS.  Another example is the permutation missingness
model of \cite{Robins97}, which for a specific
ordering of the study variables 
assumes that the probability of observing the $k$th variable 
depends on the previous study variables and the subsequent
observed variables.  The group permutation missingness model of \cite{Robins97} is an analog of the latter for groups of variables and is also NPS. 
\cite{SadinleReiter17} introduced a missingness mechanism where each variable and its missingness indicator are conditionally independent given the remaining variables and missingness indicators, which leads to a NPS model.  \cite{Tchetgenetal16} proposed a NPS approach based on discrete choice models.  Finally, we note that MAR models also can be NPS, as shown by \cite{Gilletal97}. 
\par

\section{Sequential Identification Strategy}\label{s:SeqIdent}

We consider the $p$ variables as divided into $K$ blocks,
$\bX=(\bX_1,\dots,\bX_K)\allowbreak$, where $\bX_k=(X_{t_{k-1}+1},\dots \allowbreak ,X_{t_k})$, which contains $t_k-t_{k-1}$ variables.
As our results only
concern the identification of full-data distributions starting from an
observed-data distribution, we assume that $f(\bX^*)=f(\bX_1^*,\dots,\bX_K^*)$ is known.
The identification strategy consists of specifying a sequence of assumptions $A_{1},\dots,A_{K}$, one for each block of variables, where each $A_k$ allows us to identify 
the conditional distribution of $\bX_k$ and $\bM_k$ given  $\bX_{<k}\equiv(\bX_{1},\dots,\bX_{k-1})$, $\bX_{> k}^*\equiv(\bX_{k+1}^*,\dots,\bX_{K}^*)$, and a carefully chosen subset of the missingness indicators $\bM_{<k}\equiv(\bM_{1},\dots,\bM_{k-1})$ described below.  We first provide a general description of how $A_{1},\dots,A_{K}$ allow us to identify parts of the full-data distribution in a sequential manner, and then in Theorem \ref{th:ident} present the formal identification result.

\subsection{Description}\label{ss:description}

We now present the steps needed to implement the identification strategy.  A graphical summary of the procedure is provided in Figure \ref{fig:SeqIdent}.  

{\em Step 1}. Write $f(\bX^*)=f(\bX_1^*|\bX_{>1}^*)f(\bX_{>1}^*)$. Consider an identifiability assumption $A_1$ on the distribution of $\bX_1$ and $\bM_1$ given $\bX_{>1}^*$ that allows us to obtain a distribution with density $f_{A_1}(\bX_1,\bM_1|\bX_{>1}^*)$ with the NPS property $f_{A_1}(\bX_1^*|\bX_{>1}^*)=f(\bX_1^*|\bX_{>1}^*)$.  From this we can define $f_{A_1}(\bX_1,\bM_1,\bX_{>1}^*)\equiv f_{A_1}(\bX_1,\bM_1|\bX_{>1}^*) f(\bX_{>1}^*)$.

{\em Step 2}. Suppose we divide the $t_1$ variables in $\bX_1$ into two sets indexed by $R_1$ and $S_1$, where 
$R_1\cup S_1=\{1,\dots,t_1\}$ and $R_1\cap S_1=\emptyset$; see Remark 1 below for discussion of why one might want to do so.  Let $\bM_{R_1}$ and $\bM_{S_1}$ be the corresponding missingness indicators. 
We can write
$f_{A_1}(\bX_1,\bM_1,\bX_{>1}^*)=f_{A_1}(\bM_{S_1}|\bX_1,\bM_{R_1},\bX_{>1}^*)f_{A_1}(\bX_1,\bM_{R_1},\bX_{>1}^*)$, 
where $f_{A_1}(\bX_1,\bM_{R_1},\bX_{>1}^*)=f_{A_1}(\bX_2^*|\bX_1,\bM_{R_1},\bX_{>2}^*) f_{A_1}(\allowbreak \bX_1,\bM_{R_1},\bX_{>2}^*)$.
Consider an identifiability assumption $A_2$ on the distribution of $\bX_2$ and $\bM_2$ given $\bX_1, \bM_{R_1}$ and $\bX_{>2}^*$ that allows us to obtain a distribution with density $f_{A_{\leq 2}}(\bX_2,\bM_2|\bX_1,\bM_{R_1},\bX_{>2}^*)$ with the NPS property $f_{A_{\leq 2}}(\bX_2^*|\bX_1,\bM_{R_1},\bX_{>2}^*)=f_{A_1}(\bX_2^*|\bX_1,\bM_{R_1},\bX_{>2}^*)$.  From this we can define $$f_{A_{\leq 2}}(\bX_{\leq 2},\bM_{R_1},\bM_{2},\bX_{>2}^*)\equiv f_{A_{\leq 2}}(\bX_2,\bM_2|\bX_1,\bM_{R_1},\bX_{>2}^*)f_{A_1}(\bX_1,\bM_{R_1},\bX_{>2}^*).$$  The notation $f_{A_{\leq 2}}$
emphasizes that the distribution relies on $A_1$ and $A_2$.

{\em Step $k+1$}. At the end of the $k$th step we have $f_{A_{\leq k}}(\bX_{\leq k},\bM_{R_{k-1}},\bM_{k},\bX_{>k}^*)$.  Let $R_k\cup S_k=\{t_{k-1}+1,\dots,t_k\}\cup R_{k-1}$, $R_k\cap S_k=\emptyset$, and $\bM_{R_k}$ and $\bM_{S_k}$ be the corresponding missingness indicators.  We can write $f_{A_{\leq k}}(\bX_{\leq k},\bM_{R_{k-1}},\bM_{k},\bX_{>k}^*)=f_{A_{\leq k}}(\bM_{S_k}|\bX_{\leq k},\bM_{R_{k}},\bX_{>k}^*)f_{A_{\leq k}}(\bX_{\leq k},\bM_{R_{k}},\bX_{>k}^*)$, where $f_{A_{\leq k}}(\bX_{\leq k},\bM_{R_{k}},\bX_{>k}^*)=f_{A_{\leq k}}(\bX_{k+1}^*|\bX_{\leq k},\bM_{R_{k}},\bX_{>k+1}^*)f_{A_{\leq k}}(\bX_{\leq k},\bM_{R_{k}},\bX_{>k+1}^*).$  
Now, consider an identifiability assumption $A_{k+1}$ on the distribution of $\bX_{k+1}$ and $\bM_{k+1}$ given $\bX_{\leq k}, \bM_{R_{k}}$ and $\bX_{>k+1}^*$ that allows to obtain a distribution with density $f_{A_{\leq k+1}}(\bX_{k+1},\bM_{k+1}|\bX_{\leq k},\bM_{R_{k}},\bX_{>k+1}^*)$ with the NPS property $f_{A_{\leq k+1}}(\bX_{k+1}^*|\bX_{\leq k},\bM_{R_{k}},\bX_{>k+1}^*)=f_{A_{\leq k}}(\bX_{k+1}^*|\bX_{\leq k},\bM_{R_{k}},\bX_{>k+1}^*)$.  From this we can define 
\begin{align*}
&f_{A_{\leq k+1}}(\bX_{\leq k+1},\bM_{R_{k}},\bM_{k+1},\bX_{>k+1}^*) \\
&\hspace{1.5cm}
\equiv f_{A_{\leq k+1}}(\bX_{k+1},\bM_{k+1}|\bX_{\leq k},\bM_{R_{k}},\bX_{>k+1}^*)f_{A_{\leq k}}(\bX_{\leq k},\bM_{R_{k}},\bX_{>k+1}^*).
\end{align*}

{\em Step $K$}. For the final step, assumption $A_K$ is on the distribution of $\bX_{K}$ and $\bM_{K}$ given $\bX_{<K}$ and $\bM_{R_{K-1}}$.  Following  the previous generic identifying step, we obtain 
$f_{A_{\leq K}}(\bX,\bM_{R_{K-1}},\bM_{K}) \equiv f_{A_{\leq K}}(\bX_{K},\bM_{K}|\bX_{<K},\bM_{R_{K-1}})f_{A_{<K}}(\bX_{<K},\bM_{R_{K-1}})$.
We can obtain the implied distribution of the study variables, with density $f_{A_{\leq K}}(\bX)$, from this last equation. 

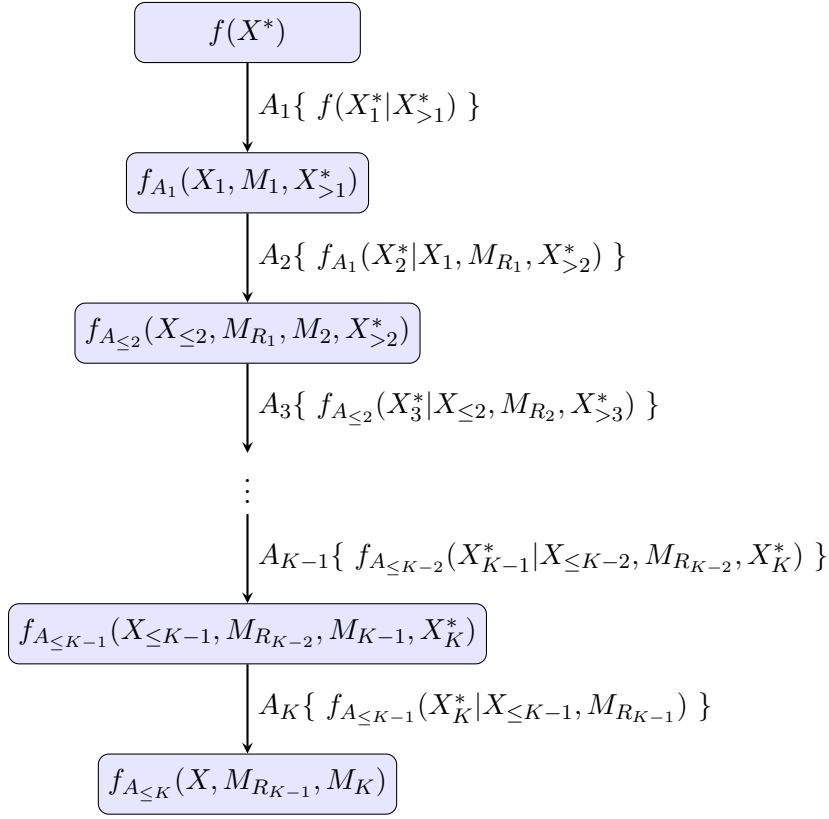
\begin{figure}[h]
\begin{center}
\begin{tikzpicture}[node distance=2cm]
\tikzstyle{startstop} = [rectangle, rounded corners, minimum width=3cm, minimum height=.8cm,text centered, draw=black, fill=blue!10]
\tikzstyle{empty} = [rectangle, rounded corners, minimum width=3cm, minimum height=.8cm,text centered]
\tikzstyle{arrow} = [thick,->,>=stealth]
\node (step0) [startstop] {$f(\bX^*)$};
\node (step1) [startstop, below of=step0] {$f_{A_1}(\bX_1,\bM_1,\bX_{>1}^*)$};
\node (step2) [startstop, below of=step1] {$f_{A_{\leq 2}}(\bX_{\leq 2},\bM_{R_1},\bM_2,\bX_{>2}^*)$};
\node (step3) [empty, below of=step2] {$\vdots$};
\node (step4) [startstop, below of=step3] {$f_{A_{\leq K-1}}(\bX_{\leq K-1},\bM_{R_{K-2}},\bM_{K-1},\bX_{K}^*)$};
\node (step5) [startstop, below of=step4] {$f_{A_{\leq K}}(\bX,\bM_{R_{K-1}},\bM_{K})$};

\draw [arrow] (step0) -- node[anchor=west] {$A_1\{~f(\bX_1^*|\bX_{>1}^*)~\}$} (step1); % $A_1\{~f(\bX_1,\bM_1|\bX_{>1}^*)~\}$
\draw [arrow] (step1) -- node[anchor=west] {$A_2\{~f_{A_1}(\bX_2^*|\bX_1,\bM_{R_1},\bX_{>2}^*)~\}$} (step2);
\draw [arrow] (step2) -- node[anchor=west] {$A_3\{~f_{A_{\leq 2}}(\bX_3^*|\bX_{\leq 2},\bM_{R_2},\bX_{>3}^*)~\}$} (step3);
\draw [arrow] (step3) -- node[anchor=west] {$A_{K-1}\{~f_{A_{\leq K-2}}(\bX_{K-1}^*|\bX_{\leq K-2},\bM_{R_{K-2}},\bX_{K}^*)~\}$} (step4);
\draw [arrow] (step4) -- node[anchor=west] {$A_K\{~f_{A_{\leq K-1}}(\bX_K^*|\bX_{\leq K-1},\bM_{R_{K-1}})~\}$} (step5);
\end{tikzpicture}
\end{center}
\caption{Sequential identification strategy.  We write $A_k\{~f_{A_{\leq k-1}}(\bX_k^*|\dots)~\}$ to indicate that assumption $A_k$ is being used to obtain a conditional full-data density $f_{A_{\leq k}}(\bX_k,\bM_k|\dots)$ from $f_{A_{\leq k-1}}(\bX_k^*|\dots)$.} \label{fig:SeqIdent}
\end{figure}

\begin{remark}  
The main characteristic of the $R_k$ subsets is that if an index does not appear in $R_{k-1}$, then it cannot appear in $R_{k}$, unless it is one of $t_{k-1}+1,\dots,t_k$.  The flexibility in the choosing of these subsets gives flexibility in the setting up of the identifiability assumptions: different versions of our identification approach can be obtained by making assumptions conditioning on different subsets of the missingness indicators.  As long as the $R_k$ subsets satisfy $R_k\subseteq \{t_{k-1}+1,\dots,t_k\}\cup R_{k-1}$, Theorem \ref{th:ident} guarantees that the final full-data distribution is NPS.
\end{remark}

\begin{remark}    The sequence for $A_1,\dots,A_K$ follows the order of the blocks $\bX_1,\dots,\bX_K$.  In many cases these blocks may not have a natural order.  Different orderings of the blocks lead to different sets of assumptions, thereby leading to different final full-data distributions and implied distributions of the study variables.  To clarify this point, suppose that we have three blocks of variables: demographic variables $\bX_D$, income variables $\bX_I$, and health-related variables $\bX_H$.  When $\bX_D$ is first in the order, $A_1$ concerns the distribution of $\bX_D$ and $\bM_D$ given $\bX_I^*$ and $\bX_H^*$; likewise, when $\bX_I$ is first in the order, $A_1$ concerns the distribution of $\bX_I$ and $\bM_I$ given $\bX_D^*$ and $\bX_H^*$.
Similarly, $A_2$ and $A_3$ also will change depending on the order of the variables, thereby implying changes in the final full-data distribution.
\end{remark}

\subsection{Non-Parametric Saturation}

The previous presentation makes it clear that the identifying assumptions $A_1,\dots,A_K$ allow us to identify $f_{A_{\leq K}}(\bX,\bM_{R_{K-1}},\bM_{K})$, and furthermore, $f_{A_{\leq k}}(\bM_{S_k}|\bX_{\leq k},\bM_{R_{k}},\bX_{>k}^*)$ for each $k<K$, although each of these conditional densities remains unused after step $k$ in the procedure.  
A full-data distribution with density $\tilde{f}_{A_{\leq K}}(\bX,\bM)$ that encodes $A_1,\dots,A_K$ can be expressed as 
$$\tilde{f}_{A_{\leq K}}(\bX,\bM)=f_{A_{\leq K}}(\bX,\bM_{R_{K-1}},\bM_{K})\tilde{f}_{A_{\leq K}}(\bM_{S_1},\dots,\bM_{S_{K-1}}|
\bX,\bM_{R_{K-1}},\bM_{K}),$$
where the second factor can be written as 
$\prod_{k=1}^{K-1}\tilde{f}_{A_{\leq K}}(\bM_{S_k}| \bX,\bM_{S_{>k}},\bM_{R_{K-1}},\bM_{K})$,
with $S_{>k}\equiv S_{k+1}\cup\dots\cup S_{K-1}$, and $\bM_{S_{>k}}\equiv (\bM_{S_{k+1}},\dots,\bM_{S_{K-1}})$.  From the definition of the sets $S_k$ and $R_k$, it is easy to see that $S_{>k}\cup R_{K-1}=R_{k}\cup\{t_k+1,\dots,t_{K-1}\}$, and therefore we can rewrite 
$\tilde{f}_{A_{\leq K}}(\bM_{S_k}| \bX,\bM_{S_{>k}},\bM_{R_{K-1}},\bM_{K})=\tilde{f}_{A_{\leq K}}(\bM_{S_k}| \bX,\bM_{R_{k}},\bM_{>k})$.

The sequential identification procedure does not identify any $\tilde{f}_{A_{\leq K}}(\bM_{S_k}| \bX,\bM_{R_{k}},\bM_{>k})$, but only $f_{A_{\leq k}}(\bM_{S_k}|\bX_{\leq k},\bM_{R_{k}},\bX_{>k}^*)$, that is, it identifies the distribution of $\bM_{S_k}$ given the variables $\bX_{\leq k}$, the missingness indicators $\bM_{R_{k}}$, and the materialized variables $\bX_{>k}^*$, but not given the missing variables among $\bX_{>k}$ according to $\bM_{>k}$.  Nevertheless, the full specification of $\tilde{f}_{A_{\leq K}}(\bM_{S_k}| \bX,\bM_{R_{k}},\bM_{>k})$ is irrelevant given that any such conditional distribution that agrees with $f_{A_{\leq k}}(\bM_{S_k}|\bX_{\leq k},\bM_{R_{k}},\bX_{>k}^*)$
would lead to the same $f_{A_{\leq K}}(\bX)$.  One such distribution is one where
\begin{equation}\label{eq:extra_assump}
\tilde{f}_{A_{\leq K}}(\bM_{S_k}| \bX,\bM_{R_{k}},\bM_{>k})=f_{A_{\leq k}}(\bM_{S_k}|\bX_{\leq k},\bM_{R_{k}},\bX_{>k}^*), ~k<K,
\end{equation}
that is, where the conditional distribution of $\bM_{S_k}$ given $\bX,\bM_{R_{k}}$ and $\bM_{>k}$ does not depend on the missing variables among $\bX_{>k}$ according to $\bM_{>k}$.  This guarantees the existence of a full-data distribution with density  
\begin{equation}\label{eq:f_tilde}
\tilde{f}_{A_{\leq K}}(\bX,\bM)=f_{A_{\leq K}}(\bX,\bM_{R_{K-1}},\bM_{K})\prod_{k=1}^{K-1}f_{A_{\leq k}}(\bM_{S_k}|\bX_{\leq k},\bM_{R_{k}},\bX_{>k}^*),
\end{equation}
which encodes the assumptions $A_1,\dots,A_K$.  Theorem \ref{th:ident} guarantees that this construction leads to NPS full-data distributions.  

\begin{theorem}\label{th:ident}
Let $R_1,\dots,R_{K-1}$ be a sequence of subsets such that $R_k\subseteq \{t_{k-1}+1,\dots,t_k\}\cup R_{k-1}$.  Let $A_1,\dots,A_K$ be a sequence of identifying assumptions, with each $A_k$ being an assumption on the conditional distribution of $\bX_k$ and $\bM_k$ given $\bX_{<k},\bM_{R_{k-1}}$, and $\bX^*_{>k}$, such that for a given density $g(\bX_{k}^*|\bX_{<k},\bM_{R_{k-1}},\bX^*_{>k})$, it allows the construction of a density $g_{A_{k}}(\bX_{k},\bM_k|\bX_{<k},\bM_{R_{k-1}},\bX^*_{>k})$ with the NPS property $g_{A_{k}}(\bX_{k}^*|\bX_{<k},\bM_{R_{k-1}},\bX^*_{>k})=g(\bX_{k}^*|\bX_{<k},\bM_{R_{k-1}},\bX^*_{>k}).$  Then, given an observed-data density $f(\bX^*)$, there exists a full-data density $\tilde{f}_{A_{\leq K}}(\bX,\bM)$ that encodes the assumptions $A_1,\dots,A_K$ and satisfies the NPS property $\tilde{f}_{A_{\leq K}}(\bX^*)=f(\bX^*)$.  
\end{theorem}
\begin{proof}  We explained how assumptions $A_1,\dots,A_K$ along with the extra assumption in \eqref{eq:extra_assump} lead to the full-data density in \eqref{eq:f_tilde}.  We now show the NPS property of \eqref{eq:f_tilde}.  
To start, we integrate \eqref{eq:f_tilde} over the missing variables in $\bX_{K}$ according to $\bM_{K}$.  Since none of the factors in
$\prod_{k=1}^{K-1}f_{A_{\leq k}}(\bM_{S_k}| \bX_{\leq k},\bM_{R_{k}},\allowbreak \bX_{>k}^*)$
depend on these missing variables, we obtain
\begin{align}
& f_{A_{\leq K-1}}(\bX_{\leq K-1},\bM_{R_{K-1}},\bX_{K}^*)\prod_{k=1}^{K-1}f_{A_{\leq k}}(\bM_{S_k}|\bX_{\leq k},\bM_{R_{k}},\bX_{>k}^*)\nonumber\\
& = f_{A_{\leq K-1}}(\bX_{\leq K-1},\bM_{R_{K-1}},\bM_{S_{K-1}},\bX_{K}^*)\prod_{k=1}^{K-2}f_{A_{\leq k}}(\bM_{S_k}|\bX_{\leq k},\bM_{R_{k}},\bX_{>k}^*)\nonumber\\
& = f_{A_{\leq K-1}}(\bX_{\leq K-1},\bM_{R_{K-2}},\bM_{K-1},\bX_{K}^*)\prod_{k=1}^{K-2}f_{A_{\leq k}}(\bM_{S_k}|\bX_{\leq k},\bM_{R_{k}},\bX_{>k}^*).\label{eq:proof1}
\end{align}

Similarly, we now integrate \eqref{eq:proof1} over the missing variables in $\bX_{K-1}$ according to $\bM_{K-1}$.  Given that none of the factors in
$\prod_{k=1}^{K-2}f_{A_{\leq k}}(\bM_{S_k}| \bX_{\leq k},\bM_{R_{k}},\bX_{>k}^*)$,
depend on these missing variables, and given the way $f_{A_{\leq K-1}}(\bX_{\leq K-1},\bM_{R_{K-2}},\bM_{K-1},\bX_{K}^*)$
is constructed (see generic step $k+1$ in Section \ref{ss:description}), we obtain 
$$f_{A_{\leq K-2}}(\bX_{\leq K-2},\bM_{R_{K-2}},\bX_{> K-2}^*)\prod_{k=1}^{K-2}f_{A_{\leq k}}(\bM_{S_k}| \bX_{\leq k},\bM_{R_{k}},\bX_{>k}^*).$$

These arguments and process can be repeated, sequentially integrating over the missing variables in $\bX_{k}$ according to $\bM_{k}$, $k=K-2,\dots,1$, finally obtaining the observed-data density $f(\bX^*)$.
\end{proof}

\subsection{Special Cases}\label{ss:extreme_cases}

It is worth describing two special sequential identification schemes that can be derived from our general presentation.  One is obtained when we take all $R_k=\{t_{k-1}+1,\dots,t_k\}\cup R_{k-1}$, $S_k=\emptyset$, and therefore $\bM_{R_k}=\bM_{\leq k}$.  In this case, each $A_{k+1}$ is on  the distribution of $\bX_{k+1}$ and $\bM_{k+1}$ given $\bX_{\leq k}, \bM_{\leq k}$ and $\bX_{>k+1}^*$, that is, the assumption conditions on the whole set of missingness indicators $\bM_{\leq k}$ and not just on a subset of these.  The other is obtained when we take all $R_k=\emptyset$, $S_k=\{t_{k-1}+1,\dots,t_k\}$, and therefore $\bM_{S_k}=\bM_{k}$.  In this case, each $A_{k+1}$ is on  the distribution of $\bX_{k+1}$ and $\bM_{k+1}$ given $\bX_{\leq k}$ and $\bX_{>k+1}^*$, that is, each assumption conditions on none of the missingness indicators $\bM_{\leq k}$.

\subsection{Connection with the Mechanisms of \cite{Robins97}}\label{ss:Robins}

An important particular case of our sequential identification strategy is obtained when all $\bM_{S_k}=\bM_{k}$ and each $A_k$ is taken to be a conditional MAR assumption, that is, when we assume that 
$f(\bM_{k}|\bX_{\leq k-1},\bX_{k},\bX_{>k}^*)=f(\bM_{k}|\bX_{\leq k-1},\bX_{k}^*,\bX_{>k}^*)$. 
Along with \eqref{eq:extra_assump}, this leads to the combined assumption
\begin{equation}\label{eq:Robins}
f(\bM_{k}| \bX,\bM_{>k})=f(\bM_{k}|\bX_{< k},\bX_{\geq k}^*).
\end{equation}
The missingness mechanism
derived from this approach corresponds to the group permutation
missingness of \cite{Robins97}. When each block contains
only one variable, it corresponds to the permutation missingness
mechanism of \cite{Robins97}.  If the ordering of the variables or blocks of variables is regarded as temporal, as in a
longitudinal study or a survey that asks questions in a fixed
sequence, \cite{Robins97} interpreted \eqref{eq:Robins} as follows: the 
nonresponse propensity at the current time period depends on the
values of study variables in the previous time periods, whether
  observed or not, but not on what is missing in the present and future time periods.

If the order of the blocks of variables was reversed, that is, if $A_1$ was on the distribution of $\bX_K$ and $\bM_K$ given $\bX_{<K}^*$,  $A_2$ was on the distribution of $\bX_{K-1}$ and $\bM_{K-1}$ given $\bX_{<K-1}^*$ and $\bX_K$, and so on, then we would have the following interpretation: the 
nonresponse propensity at the current time period depends on the
values of study variables in the future time periods, whether
  observed or not, but not on what is missing in the present and past time periods.
  This interpretation is arguably
  easier to explain in the context of 
  respondents answering a questionnaire.  The nonresponse propensity for question 
  $t$ can depend on the respondent's answers to questions that appear later in 
  the questionnaire and to questions that she has already answered, but
  not on the information that she has not revealed.

\section{Applications}\label{s:Examples}

\subsection{Sequential Identification for Two Categorical Variables}\label{s:2vars}

Consider two categorical random variables $X_1 \in \mathcal{X}_1=\{1,\dots,I\}$ and $X_2 \in \mathcal{X}_2=\{1,\dots,J\}$.  Let $M_1$ and $M_2$ be their missingness indicators.  
Let $\mathbb{P}$ denote the joint distribution of $(X_1,X_2,M_1,M_2)$.  The observed-data distribution  corresponds to the probabilities 
\begin{align*}
\pi_{ij00} & ~ \equiv ~ \mathbb{P}(X_1^*=i,X_2^*=j) ~ = ~ \mathbb{P}(X_1=i,X_2=j,M_1=0,M_2=0), \\
\pi_{i+01} & ~ \equiv ~ \mathbb{P}(X_1^*=i,X_2^*=*) ~ = ~ \mathbb{P}(X_1=i,M_1=0,M_2=1), \\
\pi_{+j10} & ~ \equiv ~ \mathbb{P}(X_1^*=*,X_2^*=j) ~ = ~ \mathbb{P}(X_2=j,M_1=1,M_2=0),\\
\pi_{++11} & ~ \equiv ~ \mathbb{P}(X_1^*=*,X_2^*=*) ~ = ~ \mathbb{P}(M_1=1,M_2=1),
\end{align*}
for $i\in\mathcal{X}_1$, $j\in\mathcal{X}_2$.  We seek to
construct a full-data distribution $\mathbb{P}_{A_{\leq 2}}(X_1,X_2,M_1,M_2)$
from the observed-data distribution $\mathbb{P}(X_1^*,X_2^*)$ by
imposing some assumptions $A_1$ and $A_2$.  
The goal is a full-data distribution such that $\mathbb{P}_{A_{\leq 2}}(X_1^*,X_2^*)=\mathbb{P}(X_1^*,X_2^*)$, that
is, we want $\mathbb{P}_{A_{\leq 2}}$ to be NPS.    

To use the general identification strategy presented in Section \ref{s:SeqIdent} we define 
each variable as its own block. With only two variables, set $R_1$ can be either $R_1=\{1\}$ or $R_1=\emptyset$. 
 We present two examples below corresponding to these two options.
\par

\begin{example}  We first consider $R_1=\{1\}$, $S_1=\emptyset$, and the following identifying assumptions:
$A_1 : ~ X_1\indep M_1| X_2^*$; and $A_2 : ~ X_2\indep M_2| M_1,X_1.$

Under $A_1$, $\mathbb{P}_{A_1}(X_1,M_1|X_2^*)=
\mathbb{P}_{A_1}(X_1|X_2^*)\mathbb{P}_{A_1}(M_1|X_2^*) = \mathbb{P}(X_1|X_2^*,M_1=0)\mathbb{P}(M_1|X_2^*)$, where $\mathbb{P}(X_1|X_2^*,M_1=0)$ and $\mathbb{P}(M_1|X_2^*)$ are identified from the observed data distribution.  When $X_2^*=j\neq *$, $\mathbb{P}(X_1=i|X_2^*=j,M_1=0)=\mathbb{P}(X_1=i|X_2=j,M_1=0,M_2=0)=\pi_{ij00}/\pi_{+j00}$, and $\mathbb{P}(M_1=m_1|X_2^*=j)=\mathbb{P}(M_1=m_1|X_2=j,M_2=0)=\pi_{+jm_10}/\pi_{+j+0}$.
Similarly, when $X^*_2=*$ we find
$\mathbb{P}(X_1=i|X_2^*=*,M_1=0)= \pi_{i+01}/\pi_{++01}$ and $\mathbb{P}(M_1=m_1|X_2^*=*) = \pi_{++m_11}/\pi_{+++1}$.  Since $\mathbb{P}(X_2^*)$ can be obtained from the observed-data distribution as $\mathbb{P}(X_2=j,M_2=0)=\pi_{+j+0}$ when $X_2^*=j\neq *$, and as $\mathbb{P}(M_2=1)=\pi_{+++1}$ when $X_2^*=*$, using $\mathbb{P}_{A_1}(X_1,M_1|X_2^*)$ we obtain a joint distribution for $(X_1,M_1,X_2^*)$ that relies on $A_1$, defined as $\mathbb{P}_{A_1}(X_1,M_1,X_2^*)\equiv \mathbb{P}_{A_1}(X_1,M_1|X_2^*)\mathbb{P}(X_2^*)$.  Note that $\mathbb{P}_{A_1}$ can be written as an explicit function of the observed-data distribution.

We now use $\mathbb{P}_{A_1}$ and identifying
assumption $A_2$ to obtain 
$\mathbb{P}_{A_{\leq 2}}(X_2,M_2|X_1,M_1).$
From the definition of $X_2^*$, $\mathbb{P}_{A_1}(X_1,M_1,X_2^*)$ can be written as $\mathbb{P}_{A_1}(X_1,M_1,X_2,M_2=0)$ when $X_2^*\neq *$ and $\mathbb{P}_{A_1}(X_1,M_1,M_2=1)$ when $X_2^*=*$.  From this we can obtain 
$$\mathbb{P}_{A_1}(M_2 = 1 | X_1,M_1)=\frac{\mathbb{P}_{A_1}(X_1,M_1,M_2=1)}{\mathbb{P}_{A_1}(X_1,M_1,M_2=1)+\sum_{x_2\in\mcX_2}\mathbb{P}_{A_1}(X_1,M_1,X_2=x_2,M_2=0)},$$
and 
$$\mathbb{P}_{A_1}(X_2|X_1,M_1,M_2=0)=\frac{\mathbb{P}_{A_1}(X_1,M_1,X_2,M_2=0)}{\sum_{x_2\in\mcX_2}\mathbb{P}_{A_1}(X_1,M_1,X_2=x_2,M_2=0)}.$$
We then obtain 
$\mathbb{P}_{A_{\leq 2}}(X_2,M_2|X_1,M_1) %\\\nonumber
 = \mathbb{P}_{A_{\leq 2}}(X_2|X_1,M_1)\mathbb{P}_{A_{\leq 2}}(M_2|X_1,M_1) %\\
 = \mathbb{P}_{A_1}(X_2|X_1,M_1,\allowbreak M_2=0)\mathbb{P}_{A_1}(M_2|X_1,M_1)$, %\label{eq:st2FSI_CI}
which gives us a way to obtain
$\mathbb{P}_{A_{\leq 2}}(X_2,M_2|X_1,M_1)$ as a function of the distribution
$\mathbb{P}_{A_1}$, which in turn is a function of the observed-data
distribution.  The final full-data distribution is obtained as $\mathbb{P}_{A_{\leq 2}}(X_1,M_1,X_2,M_2)\equiv \mathbb{P}_{A_{\leq 2}}(X_2,M_2|X_1,M_1)\mathbb{P}_{A_{1}}(X_1,M_1)$, where $\mathbb{P}_{A_{1}}(X_1,M_1)$ can be obtained from $\mathbb{P}_{A_{1}}$.
After some algebra we find
\begin{align*}\nonumber
\mathbb{P}_{A_{\leq 2}}(X_1=i,X_2=j,M_1=m_1,M_2=m_2)
 & = \frac{\frac{\pi_{ij00}}{\pi_{+j00}}\pi_{+jm_10}}{(\sum_l\frac{\pi_{il00}}{\pi_{+l00}}\pi_{+lm_10})^{m_2}}
\left(\frac{\pi_{i+01}}{\pi_{++01}}\pi_{++m_11}\right)^{m_2}.
\end{align*}
It is easy to see that $\mathbb{P}_{A_{\leq 2}}$ is NPS, that is $\mathbb{P}_{A_{\leq 2}}(X_1^*,X_2^*)=\mathbb{P}(X_1^*,X_2^*)$.
From the final distribution $\mathbb{P}_{A_{\leq 2}}(X_1,X_2,M_1,M_2)$ we can now obtain 
\begin{align}\label{eq:ttP}\nonumber
&\mathbb{P}_{A_{\leq 2}}(X_1=i,X_2=j) \\
& \hspace{.5cm}=\pi_{ij00}+\pi_{i+01}\frac{\pi_{ij00}}{\pi_{i+00}}+\pi_{+j10}\frac{\pi_{ij00}}{\pi_{+j00}}+
\pi_{++11}\frac{\pi_{i+01}}{\pi_{++01}}\frac{\frac{\pi_{ij00}}{\pi_{+j00}}\pi_{+j10}}{\sum_{l}\frac{\pi_{il00}}{\pi_{+l00}}\pi_{+l10}},
\end{align}
which is the distribution of inferential interest.  

In closing this example, we stress that the final full-data distribution
is not invariant to the order in which the blocks of variables appear in the sequence of assumptions.  From expression \eqref{eq:ttP} it is clear that
the final distribution of the study variables would be different had we
identified a distribution for $(X_1^*,X_2,M_2)$ first.  Indeed, if we were to follow the steps in the previous example but reversing the order of the variables, then we would be assuming that 
$X_2\indep M_2| X_1^*$ and 
$X_1\indep M_1| M_2,X_2$, 
which are different from $A_1$ and $A_2$ in this example.% in \eqref{a:ex1}.
\end{example}
  
\par

\begin{example} We now consider $R_1=\emptyset$, $S_1=\{1\}$, and the  identifying assumptions
$B_1 : ~  X_1\indep M_1| X_2^*$, and
$B_2 : ~  X_2\indep M_2| X_1$.

Assumption $B_1$ is the same as $A_1$ in Example 1, and so $\mathbb{P}_{B_1}(X_1,M_1,X_2^*)=\mathbb{P}_{A_1}(X_1,M_1,X_2^*)$.  Assumption $B_2$ is made conditioning only on $X_1$, so we need to marginalize over $M_1$ to obtain $\mathbb{P}_{B_1}(X_1,X_2^*)=\mathbb{P}_{B_1}(X_1,M_1=0,X_2^*)+\mathbb{P}_{B_1}(X_1,M_1=1,X_2^*)$:
\begin{align*}\nonumber
\mathbb{P}_{B_1}(X_1=i,X_2^*=j) & = \mathbb{P}_{B_1}(X_1=i,X_2=j,M_2=0)  = \frac{\pi_{ij00}}{\pi_{+j00}}\pi_{+j+0},\\
\mathbb{P}_{B_1}(X_1=i,X_2^*=*) & = \mathbb{P}_{B_1}(X_1=i,M_2=1)  = \frac{\pi_{i+01}}{\pi_{++01}}\pi_{+++1}.
\end{align*}
From this we can obtain 
$$\mathbb{P}_{B_1}(M_2 = 1 | X_1)=\frac{\mathbb{P}_{B_1}(X_1,M_2=1)}{\mathbb{P}_{B_1}(X_1,M_2=1)+\sum_{x_2\in\mcX_2}\mathbb{P}_{B_1}(X_1,X_2=x_2,M_2=0)},$$
and
$$\mathbb{P}_{B_1}(X_2|X_1,M_2=0)=\frac{\mathbb{P}_{B_1}(X_1,X_2,M_2=0)}{\sum_{x_2\in\mcX_2}\mathbb{P}_{B_1}(X_1,X_2=x_2,M_2=0)}.$$

Using assumption $B_2$, we obtain $\mathbb{P}_{B_{\leq 2}}(X_2,M_2| X_1) = \mathbb{P}_{B_{\leq 2}}(X_2| X_1)\mathbb{P}_{B_{\leq 2}}(M_2| X_1) = \mathbb{P}_{B_1}(X_2| X_1,\allowbreak M_2=0)\mathbb{P}_{B_1}(M_2| X_1)$. 
From this we obtain $\mathbb{P}_{B_{\leq 2}}(X_1,X_2,M_2)\equiv \mathbb{P}_{B_{1}}(X_1)\mathbb{P}_{B_{\leq 2}}(X_2,M_2|X_1)$ as
\begin{align}\nonumber
\mathbb{P}_{B_{\leq 2}}(X_1=i,X_2=j,M_2=m_2) 
& = \frac{\frac{\pi_{ij00}}{\pi_{+j00}}\pi_{+j+0}}{(\sum_l\frac{\pi_{il00}}{\pi_{+l00}}\pi_{+l+0})^{m_2}}\left(\frac{\pi_{i+01}}{\pi_{++01}}\pi_{+++1}\right)^{m_2}.
\end{align}
Marginalizing over $M_2$, we get
\begin{align}\nonumber
\mathbb{P}_{B_{\leq 2}}(X_1=i,X_2=j) & = \pi_{+j+0}\frac{\pi_{ij00}}{\pi_{+j00}}
+\pi_{+++1}\frac{\pi_{i+01}}{\pi_{++01}}\frac{\frac{\pi_{ij00}}{\pi_{+j00}}\pi_{+j+0}}{\sum_l\frac{\pi_{il00}}{\pi_{+l00}}\pi_{+l+0}}.
\end{align}

Assumptions $B_1$ and $B_2$ are enough to identify $\mathbb{P}_{B_{\leq 2}}(X_1,X_2,M_2)$, and thereby a distribution of the study variables $\mathbb{P}_{B_{\leq 2}}(X_1,X_2)$.  Although irrelevant for obtaining the distribution of the study variables, it is worth noticing that $B_1$ and $B_2$ do not allow us to fully identify
$\mathbb{P}_{B_{\leq 2}}(M_1|X_1,X_2,M_2)$.
From $\mathbb{P}_{B_1}(X_1,M_1,X_2^*)$ we have 
$\mathbb{P}_{B_1}(M_1|X_1,X_2,M_2=0)=\mathbb{P}(M_1|X_2,M_2=0)$ and 
$\mathbb{P}_{B_1}(M_1|X_1,M_2=1)=\mathbb{P}(M_1|M_2=1)$, 
but $\mathbb{P}_{B_{\leq 2}}(M_1|X_1,X_2,M_2=1)$ remains unidentified.  A full-data distribution $\tilde{\mathbb{P}}_{B_{\leq 2}}$ becomes identified under the extra assumption $\tilde{\mathbb{P}}_{B_{\leq 2}}(M_1|X_1,X_2,M_2=1)=\tilde{\mathbb{P}}_{B_{\leq 2}}(M_1|X_1,M_2=1),$
which corresponds to the extra assumption in \eqref{eq:extra_assump}.

The set of assumptions that we used in this example can be summarized in terms of
the missingness mechanism $\tilde{\mathbb{P}}_{B_{\leq 2}}(M_1,M_2|
X_1,X_2)=\tilde{\mathbb{P}}_{B_{\leq 2}}(M_1|
X_1,X_2,M_2)\tilde{\mathbb{P}}_{B_{\leq 2}}(M_2| X_1,X_2),$ where 
$\tilde{\mathbb{P}}_{B_{\leq 2}}(M_1|\allowbreak X_1,X_2,M_2=1)  =\mathbb{P}(M_1|M_2=1)$, %\\
$\tilde{\mathbb{P}}_{B_{\leq 2}}(M_1|X_1,X_2,M_2=0)  =\mathbb{P}(M_1|X_2,M_2=0)$, and %\\
$\tilde{\mathbb{P}}_{B_{\leq 2}}(M_2|\allowbreak X_1,X_2) =\mathbb{P}_{B_1}(M_2|X_1)$.
This corresponds to the permutation missingness (PM) mechanism of \cite{Robins97}.
  
As in Example 1, the full-data distribution changes when we modify the order in which the blocks of variables appear in the identifying assumptions.  Changing the order of the variables in this example would correspond to making the assumptions 
$X_2\indep M_2| X_1^*$ and 
$X_1\indep M_1| X_2$.
\end{example}
\par

\subsection{Sequential Identification for Partially Ignorable Mechanisms}\label{s:PartialIgno}

\cite{HarelSchafer09} introduced different notions of the missing data being {\em partially ignorable}.  In particular, in some scenarios one may think that the missingness is ignorable for some, but not for all the variables.  For example, consider a survey with two blocks of items $\bX_S$ and $\bX_N$, which contain responses to sensitive and non-sensitive questions, respectively.  Given the nature of these variables, one may think that the missingness among the $\bX_N$ variables could be ignored, but not among  $\bX_S$.  Our sequential identification procedure can be used to guarantee  identifiability under such partially ignorable mechanisms.  
 Our goal here is to show that we can identify a NPS full-data distribution $\tilde{f}_{A_{\leq 2}}(\bX_S,\bX_N,\bM_S,\bM_N)$ with the property that the missingness mechanism for $\bX_N$ is partially MAR given $\bM_S$ (\cite{HarelSchafer09}), that is,
\begin{equation}\label{as:PMAR}
\tilde{f}_{A_{\leq 2}}(\bM_N|\bX_N,\bX_S,\bM_S)=\tilde{f}_{A_{\leq 2}}(\bM_N|\bX_N^*,\bX_S^*),
\end{equation}
while $\tilde{f}_{A_{\leq 2}}(\bM_S|\bX_N,\bX_S)$ is determined by some nonignorable assumption.  
As before, we consider $f(\bX_S^*,\bX_N^*)$ to be known.

Following our sequential identification procedure, we first consider an identifying assumption for the distribution of $\bX_N$ and $\bM_N$ given $\bX_S^*$.  We use the conditional MAR assumption:
\begin{equation}\label{eq:A1PMAR}
A_1:~ f(\bM_N|\bX_N,\bX_S^*)=f(\bM_N|\bX_N^*,\bX_S^*).
\end{equation}
This assumption guarantees the existence of a distribution of the variables $\bX_N,\bM_N$, and $\bX_S^*$ with density 
$f_{A_1}(\bX_N,\bM_N,\bX_S^*) =f_{A_1}(\bM_N|\bX_N,\bX_S^*)f_{A_1}(\bX_N,\bX_S^*) 
\equiv f(\bM_N|\bX_N^*,\bX_S^*)f_{A_1}(\bX_N|\bX_S^*)f(\bX_S^*)$,
where $f_{A_1}(\bX_N|\bX_S^*)$ can be obtained from $f(\bX_N^*|\bX_S^*)$ as described in page 28 of \cite{Robins97}.  

Taking $R_1=\emptyset$ in our identification procedure, we can now consider any identifying assumption, say $A_2$, for the distribution of $\bX_S$ and $\bM_S$ given $\bX_N$ that allows us to obtain $f_{A_{\leq 2}}(\bX_S,\bM_S|\bX_N)$ with the NPS property $f_{A_{\leq 2}}(\bX_S^*|\bX_N)=f_{A_{1}}(\bX_S^*|\bX_N)$.  For example, $A_2$ could come from one of the approaches mentioned in Section \ref{ss:NPS}.  We then define $f_{A_{\leq 2}}(\bX_N,\bX_S,\bM_S)\equiv f_{A_{\leq 2}}(\bX_S,\bM_S|\bX_N)f_{A_1}(\bX_N)$.

To fully identify a full-data distribution $\tilde{f}_{A_{\leq 2}}(\bX_S,\bX_N,\bM_S,\bM_N)$ that encodes assumptions $A_1$ and $A_2$, we further require the conditional missingness mechanism $\tilde{f}_{A_{\leq 2}}(\bM_N|\bX_N,\bX_S,\bM_S)$.  Under the extra assumption 
\begin{equation}\label{eq:PMARextra}
\tilde{f}_{A_{\leq 2}}(\bM_N|\bX_N,\bX_S,\bM_S)=\tilde{f}_{A_{\leq 2}}(\bM_N|\bX_N,\bX_S^*),
\end{equation}
and then using $A_1$ we have identified a full-data distribution with density 
$f(\bM_N|\bX_N^*,\bX_S^*)f_{A_{\leq 2}}(\bX_N,\allowbreak \bX_S,\bM_S)$.
The NPS property of this distribution is guaranteed by Theorem \ref{th:ident}.  

A possibility for the $A_2$ assumption could come from the itemwise conditionally independent nonresponse (ICIN) mechanism of \cite{SadinleReiter17}, which is NPS.  Denoting $\bX_{S}=(X_{S1},\dots,X_{Sp_S})$, the ICIN assumption for $\bX_S$ and $\bM_S$ given $\bX_N$ can be written as
\begin{equation}\label{as:condICIN}
X_{Sj}\indep M_{Sj}| \bX_{S(-j)}, \bM_{S(-j)}, \bX_N; ~ j=1,\dots,p_S;
\end{equation}
where $\bX_{S(-j)}$ is the vector obtained from removing the $j$th entry of $\bX_{S}$, likewise for $\bM_{S(-j)}$.  Our sequential identification procedure guarantees that assumptions in \eqref{as:PMAR} and \eqref{as:condICIN} jointly  identify a NPS full-data distribution.

\begin{example}\label{ex:pim} For simplicity, consider $\bX_N=X_1$ and $\bX_S=(X_2,X_3)$.  The observed-data density can be written as the product of the density of the observed variables given each missingness pattern times the probability of the missingness pattern, that is
$f(\bX_{\bar{\bm}},\bM=\bm)=f_{\bm}(\bX_{\bar{\bm}})\pi_\bm$, which for three variables is given by 
$f_{000}(X_1,X_2,X_3)\pi_{000}$, 
$f_{100}(X_2,X_3)\pi_{100}$, $f_{010}(X_1,X_3)\pi_{010}$, $\dots$,  
$f_{011}(X_1)\pi_{011}$, 
and $\pi_{111}$.  Assumption $A_1$ in \eqref{eq:A1PMAR} in this case becomes
$A_1:~ X_1\indep M_1| X_2^*, X_3^*$,
which for all $x_2\in\mathcal{X}_2$ and $x_3\in\mathcal{X}_3$ can be expanded as
$X_1\indep M_1| M_2=0,M_3=0,X_2=x_2,X_3=x_3$;
$X_1\indep M_1| M_2=1,M_3=0,X_3=x_3$;
$X_1\indep M_1| M_2=0,M_3=1,X_2=x_2$; and 
$X_1\indep M_1| M_2=1,M_3=1$.
Using $A_1$ and the observed-data distribution we can obtain $f_{A_1}(X_1,M_1|X_2^*,X_3^*)$ as 
$f_{A_1}(X_1,M_1|X_2^*,X_3^*)=f_{A_1}(X_1|X_2^*,X_3^*)f_{A_1}(M_1|X_2^*,X_3^*)
=f(X_1|M_1=0,X_2^*,X_3^*)f(M_1|X_2^*,X_3^*)$,
where $f(X_1|M_1=0,X_2^*,X_3^*)$ is obtained from
\begin{align*}
f(X_1|M_1=0,M_2=0,M_3=0,X_2,X_3)&=f_{000}(X_1,X_2,X_3)/f_{000}(X_2,X_3),\\
f(X_1|M_1=0,M_2=1,M_3=0,X_3)&=f_{010}(X_1,X_3)/f_{010}(X_3),\\
f(X_1|M_1=0,M_2=0,M_3=1,X_2)&=f_{001}(X_1,X_2)/f_{001}(X_2),
\end{align*}
$f(X_1|M_1=0,M_2=1,M_3=1)=f_{011}(X_1)$; and $f(M_1|X_2^*,X_3^*)$ from
\begin{align*}
f(M_1=m_1|M_2=M_3=0,X_2,X_3)&\propto [f_{000}(X_2,X_3)\pi_{000}]^{I(m_1=0)}[f_{100}(X_2,X_3)\pi_{100}]^{I(m_1=1)},\\
f(M_1=m_1|M_2=1,M_3=0,X_3)&\propto [f_{010}(X_3)\pi_{010}]^{I(m_1=0)}[f_{110}(X_3)\pi_{110}]^{I(m_1=1)},\\
f(M_1=m_1|M_2=0,M_3=1,X_2)&\propto [f_{001}(X_2)\pi_{001}]^{I(m_1=0)}[f_{101}(X_2)\pi_{101}]^{I(m_1=1)},\\
f(M_1=m_1|M_2=1,M_3=1)&\propto [\pi_{011}]^{I(m_1=0)}[\pi_{111}]^{I(m_1=1)}.
\end{align*}
From this we can define $f_{A_1}(X_1,M_1,X_2^*,X_3^*)\equiv f_{A_1}(X_1,M_1|X_2^*,X_3^*)f(X_2^*,X_3^*)$, where $f(X_2^*,X_3^*)$ is obtained from
$f(M_2=0,M_3=0,X_2,X_3)=f_{000}(X_2,X_3)\pi_{000} + f_{100}(X_2,X_3)\pi_{100}$,
$f(M_2=1,M_3=0,X_3)=f_{010}(X_3)\pi_{010}+f_{110}(X_3)\pi_{110}$, %,\\
$f(M_2=0,M_3=1,X_2)=f_{001}(X_2)\pi_{001}+f_{101}(X_2)\pi_{101}$, and %\\
$f(M_2=1,M_3=1)=\pi_{011}+\pi_{111}$.

We now incorporate the ICIN assumption for the distribution of $(X_2,X_3,M_2,M_3)$ given $X_1$. We have 
$A_2: \{ X_2\indep M_2| X_3,M_3,X_1$; and $X_3\indep M_3| X_2,M_2,X_1\}$.
The identification results of \cite{SadinleReiter17} guarantee that assumption $A_2$ leads to a conditional distribution $f_{A_{\leq 2}}(X_2,X_3,M_2,M_3|X_1)$ with the NPS property $f_{A_{\leq 2}}(X_2^*,X_3^*|X_1)=f_{A_{1}}(X_2^*,X_3^*|X_1)$, where
$f_{A_{1}}(X_2^*,\allowbreak X_3^*| X_1)$ can be obtained easily from $f_{A_1}(X_1,M_1,X_2^*,X_3^*)$.  
Section 5.1 of \cite{SadinleReiter17} provides explicit formulae for the full-data distribution under the ICIN assumption as a function of the observed-data distribution, in the case of two variables.  
We can use those formulae here with $f_{A_{1}}(X_2^*,X_3^*|X_1)$ to obtain conditional ICIN full-data distributions that depend on $X_1$.  
To simplify the notation below we replace $f_{A_1}$ by $g$, and $f_{A_{\leq 2}}$ by $h$, and we denote $g_{m_2m_3}(X_2,X_3|X_1)\equiv g(X_2,X_3|X_1,M_2=m_2,M_3=m_3)$ and $h_{m_2m_3}(X_2,X_3|X_1)\equiv h(X_2,X_3|X_1,M_2=m_2,M_3=m_3)$.  
Following the formulae of \cite{SadinleReiter17} we obtain $h_{00}(X_2,X_3|X_1)=g_{00}(X_2,X_3|X_1)$, $h_{01}(X_2,X_3|X_1)=g_{00}(X_2,X_3|X_1)g_{01}(X_2|X_1)/g_{00}(X_2|X_1)$, $h_{10}(X_2,X_3|X_1)=g_{00}(X_2,X_3|X_1)g_{10}(X_3|\allowbreak X_1)/g_{00}(X_3|X_1)$,
\begin{align*}
h_{11}(X_2,X_3|X_1)&\propto \frac{g_{00}(X_2,X_3|X_1)g_{01}(X_2|X_1)g_{10}(X_3|X_1)}{g_{00}(X_2|X_1)g_{00}(X_3|X_1)},
\end{align*}
and $h(M_2=m_2,M_3=m_3|X_1)=g(M_2=m_2,M_3=m_3|X_1)$.

Putting everything together, we obtain
\begin{align*}
f_{A_{\leq 2}}(X_1,X_2,X_3,M_2=m_2,M_3=m_3)&=h_{m_2m_3}(X_2,X_3|X_1)g(M_2=m_2,M_3=m_3,X_1),
\end{align*}
from which we can obtain the distribution of the study variables $f_{A_{\leq 2}}(X_1,X_2,X_3)$.  A full-data density $\tilde{f}_{A_{\leq 2}}(X_1,X_2,X_3,M_1,M_2,M_3)$ becomes identified under the extra assumption \eqref{eq:PMARextra}.  This distribution therefore encodes the partial ignorability assumption \eqref{as:PMAR} for the missingness in $X_1$ and the ICIN assumption \eqref{as:condICIN} for $(X_2,X_3,M_2,M_3)$ given $X_1$.
\end{example}

\subsection{Usage in Sensitivity Analysis}\label{s:sensit}

To illustrate how this approach could be used for sensitivity analysis, we use data related to the 1991 plebiscite where Slovenians 
voted for independence from Yugoslavia (\cite{Rubinetal95}).  The data come from the Slovenian public opinion survey, which contained the  questions:
$X_I$: are you in favor of Slovenia's independence? $X_S$: are you in favor of Slovenia's secession from Yugoslavia? $X_A$: will you attend the plebiscite?
We call these the Independence, Secession, and Attendance questions, respectively.
The possible responses to each of these were \textsc{yes}, \textsc{no}, and \textsc{don't know}.  We follow \cite{Rubinetal95} in treating \textsc{don't know} as missing data. 

We 
use the missingness mechanism presented in Example \ref{ex:pim}, and compare it with an ignorable approach, a pattern mixture model (PMM) under the complete-case missing-variable restriction (\cite{Little93}), and the ICIN approach of \cite{SadinleReiter17} which here corresponds to assuming 
$X_j\indep M_j | \bX_{-j},\bM_{-j}; ~ j=I,S,A$.
The Attendance question is arguably the less sensitive of the three questions studied here, so it seems reasonable to consider a partially ignorable mechanism where the nonresponse for $X_A$ is ignorable given $M_I$ and $M_S$, as in \eqref{as:PMAR}, and the nonresponse for $X_I$ and $X_S$ satisfy the ICIN assumption conditioning on $X_A$, as in \eqref{as:condICIN} in Example \ref{ex:pim}.  Nothing prevents us from using this approach exchanging the roles of the variables, so we also consider two other partially ignorable  missingness mechanisms, depending on whether we take the nonresponse for $X_I$ or for $X_S$ as ignorable.

To implement these approaches, we first use a Bayesian approach to estimate the observed-data distribution.  The observed data can be organized in a three-way contingency table with cells corresponding to each element of $\{$\textsc{yes}, \textsc{no}, \textsc{don't know}$\}^3$, as presented in \cite{Rubinetal95}.  Treating these data as a random sample from a multinomial distribution, we take a conjugate prior distribution for the cell probabilities: symmetric Dirichlet with parameter $1/27$.  
 We take 5,000 draws from the posterior distribution of the observed-data distribution, and for each of these we apply the formulae presented in Example 3 to obtain posterior draws of the full-data distribution under each of the three partially ignorable mechanisms.  We use a similar approach to obtain posterior draws of the full-data distribution under ICIN, PMM, and ignorability.  For each of the approaches we  then obtain draws of the implied probabilities for the items.

Figure \ref{fig:Slovenia} displays 5,000 draws from the joint posterior distribution of $\mathbb{P}$(Independence = \textsc{yes}, Attendance = \textsc{yes}) and $\mathbb{P}$(Attendance = \textsc{no}) under each of the six missingness mechanisms considered here. 
 Despite the fact that all of these approaches agree in their fit to the observed data, we obtain quite different inferences under each assumption. 
 When inferences are so sensitive to the identifying assumptions, perhaps the most honest way to proceed is to report all the results under all assumptions deemed plausible given the context.  

\begin{figure}[h]
         \begin{subfigure}{0.32\textwidth}
                 \centering
								~~~~Ignorability \\\vspace{3mm}
								%\caption{}
                 \label{fig:SlovMAR}\vspace{-.3cm}
                 \includegraphics[width=1\textwidth]{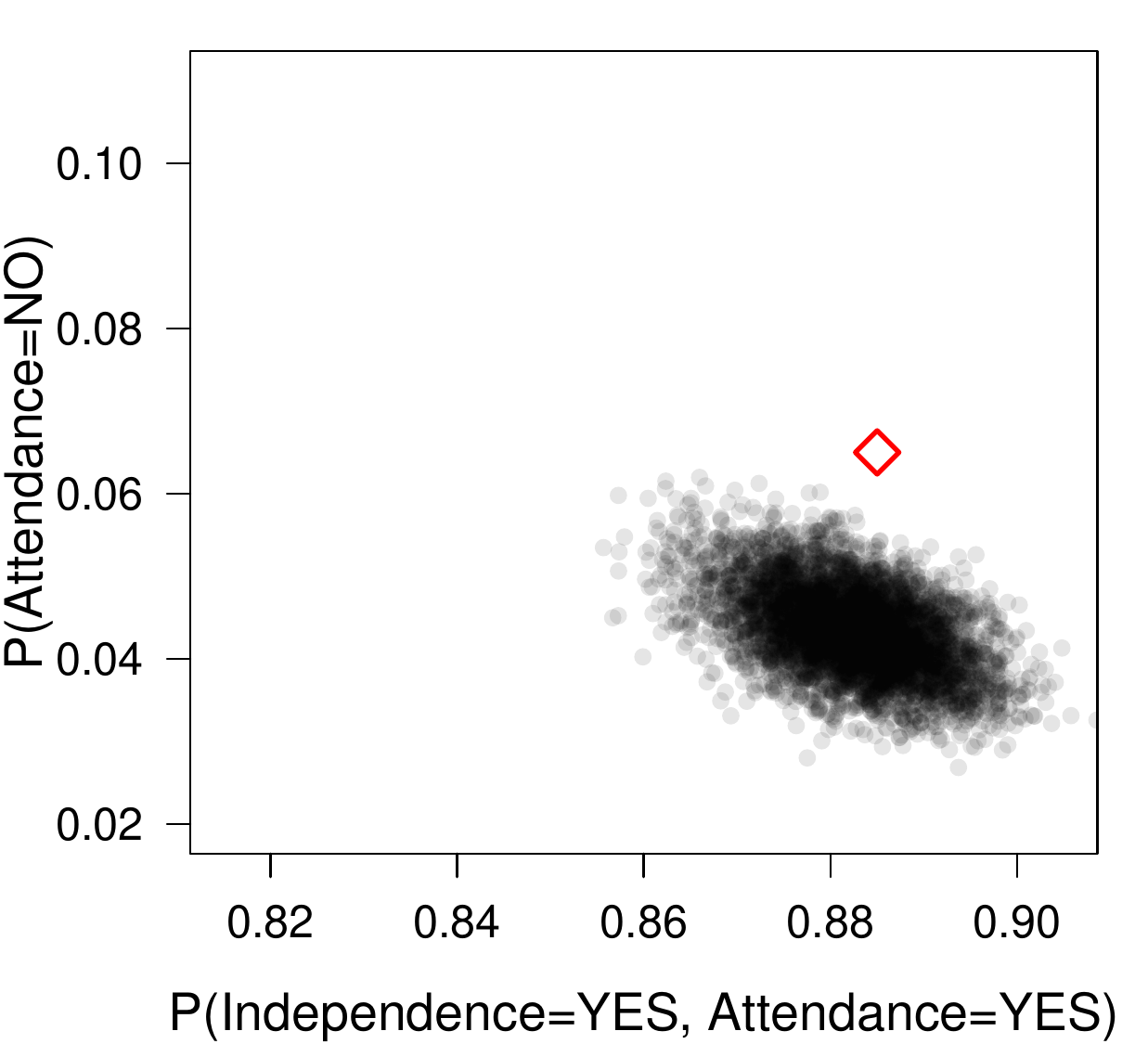}
         \end{subfigure}
     \begin{subfigure}{0.32\textwidth}
             \centering
						~~~Pattern Mixture \\\vspace{3mm}
						%\caption{}
             \label{fig:SlovPMM}\vspace{-.3cm}
             \includegraphics[width=1\textwidth]{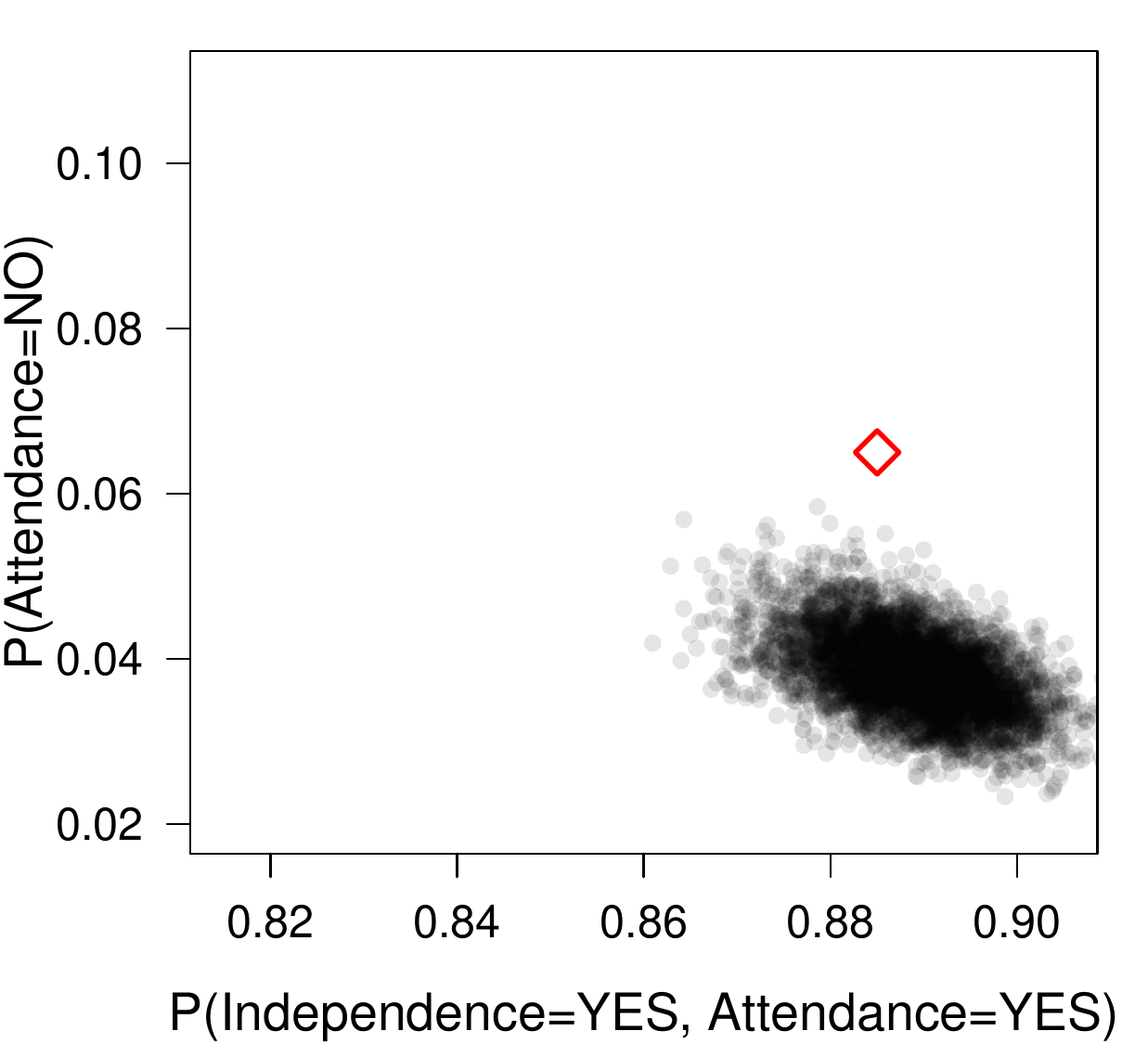}
     \end{subfigure}
		        \begin{subfigure}{0.32\textwidth}
								\centering
								~~~ICIN \\\vspace{3mm}%\caption{}MAR
                 \label{fig:SlovIMAR}\vspace{-.3cm}
                 \includegraphics[width=1\textwidth]{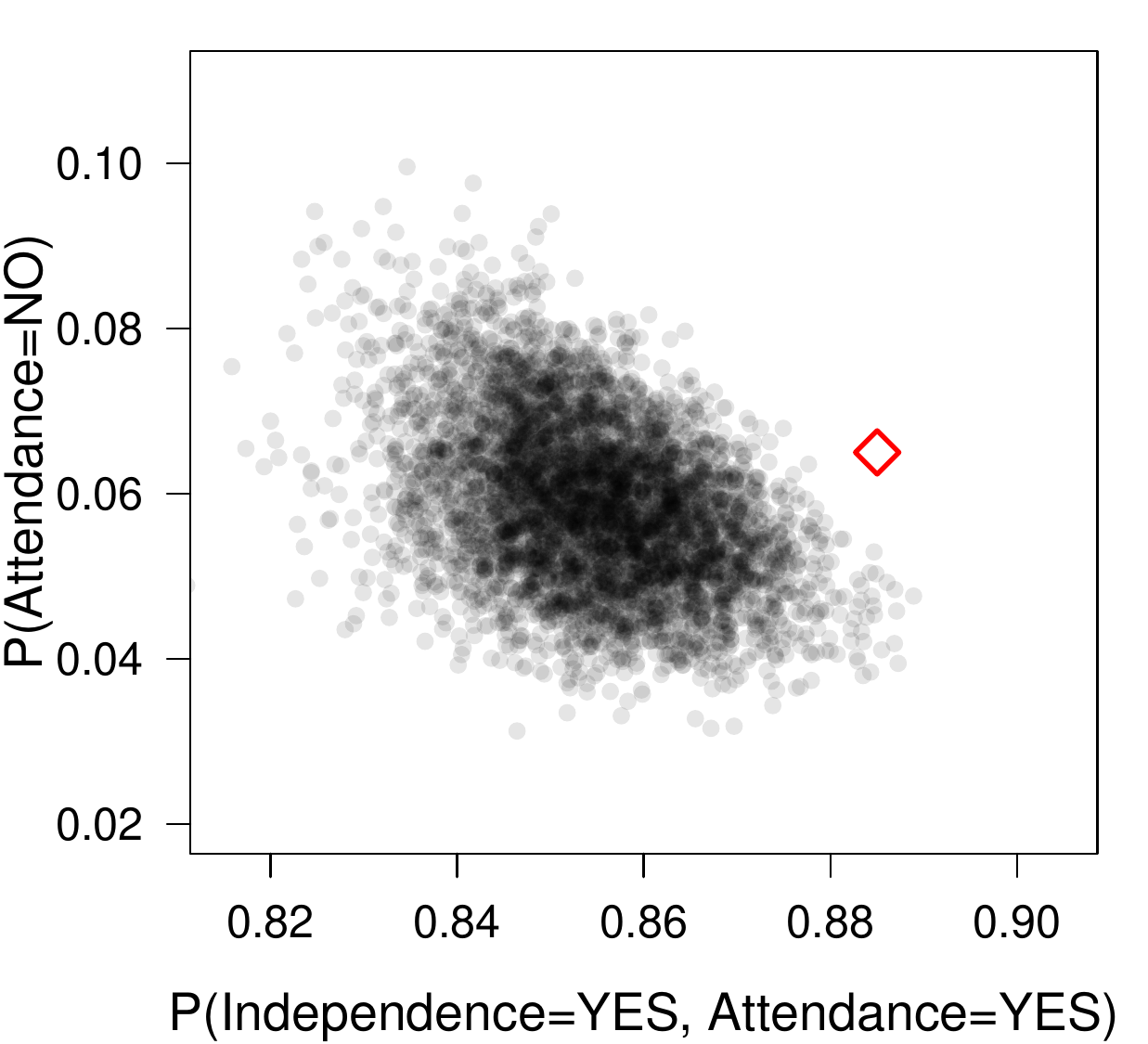}
         \end{subfigure}\vspace{2mm}\\
		\begin{subfigure}{0.32\textwidth}
             \centering
						~~~~PIM - Attendance \\\vspace{3mm}
						%\caption{}
             \label{fig:SlovPIM_Att}\vspace{-.3cm}
             \includegraphics[width=1\textwidth]{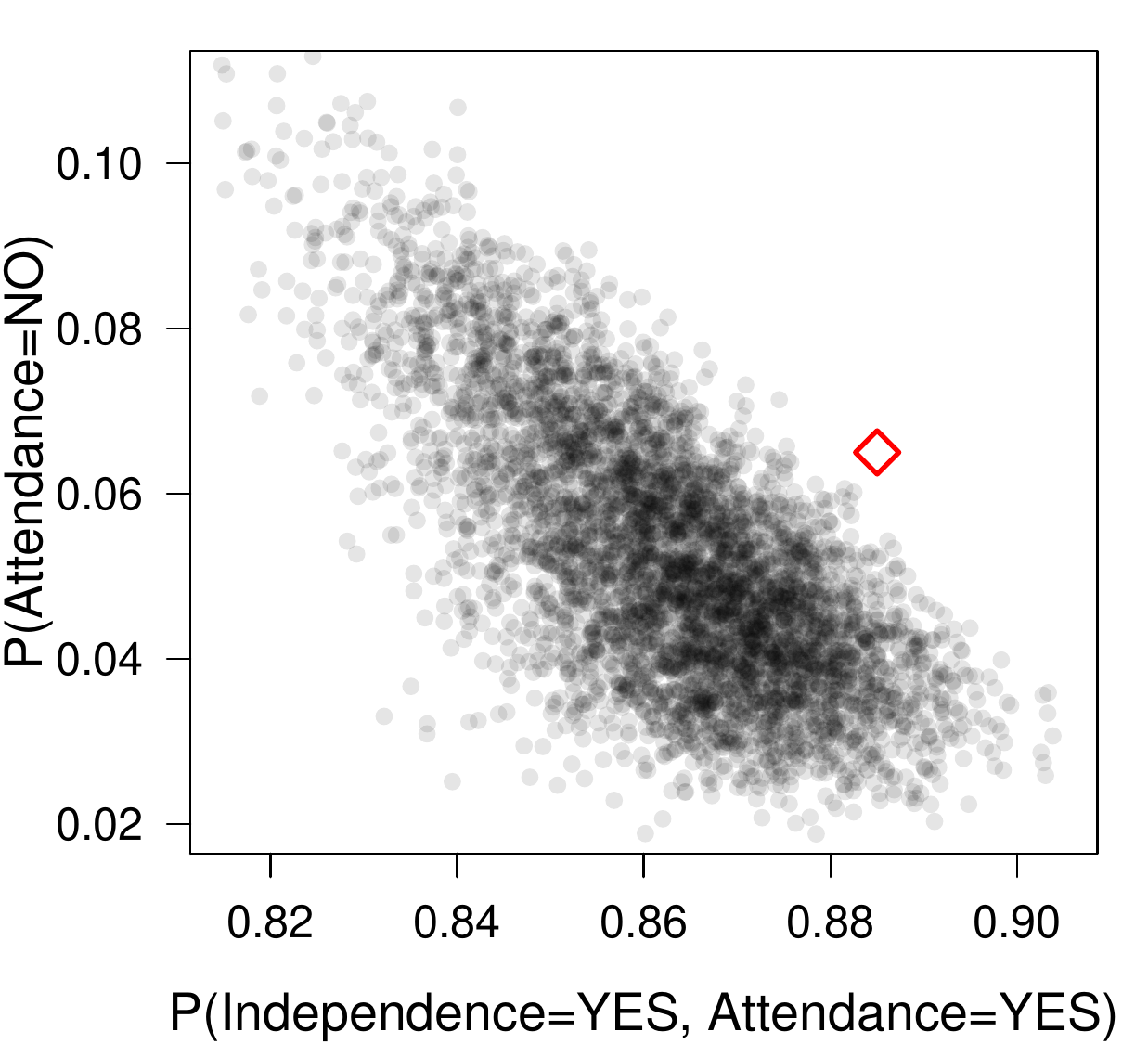}
     \end{subfigure}
		\begin{subfigure}{0.32\textwidth}
             \centering
						~~~PIM - Independence \\\vspace{3mm}
						%\caption{}
             \label{fig:SlovPIM_Ind}\vspace{-.3cm}
             \includegraphics[width=1\textwidth]{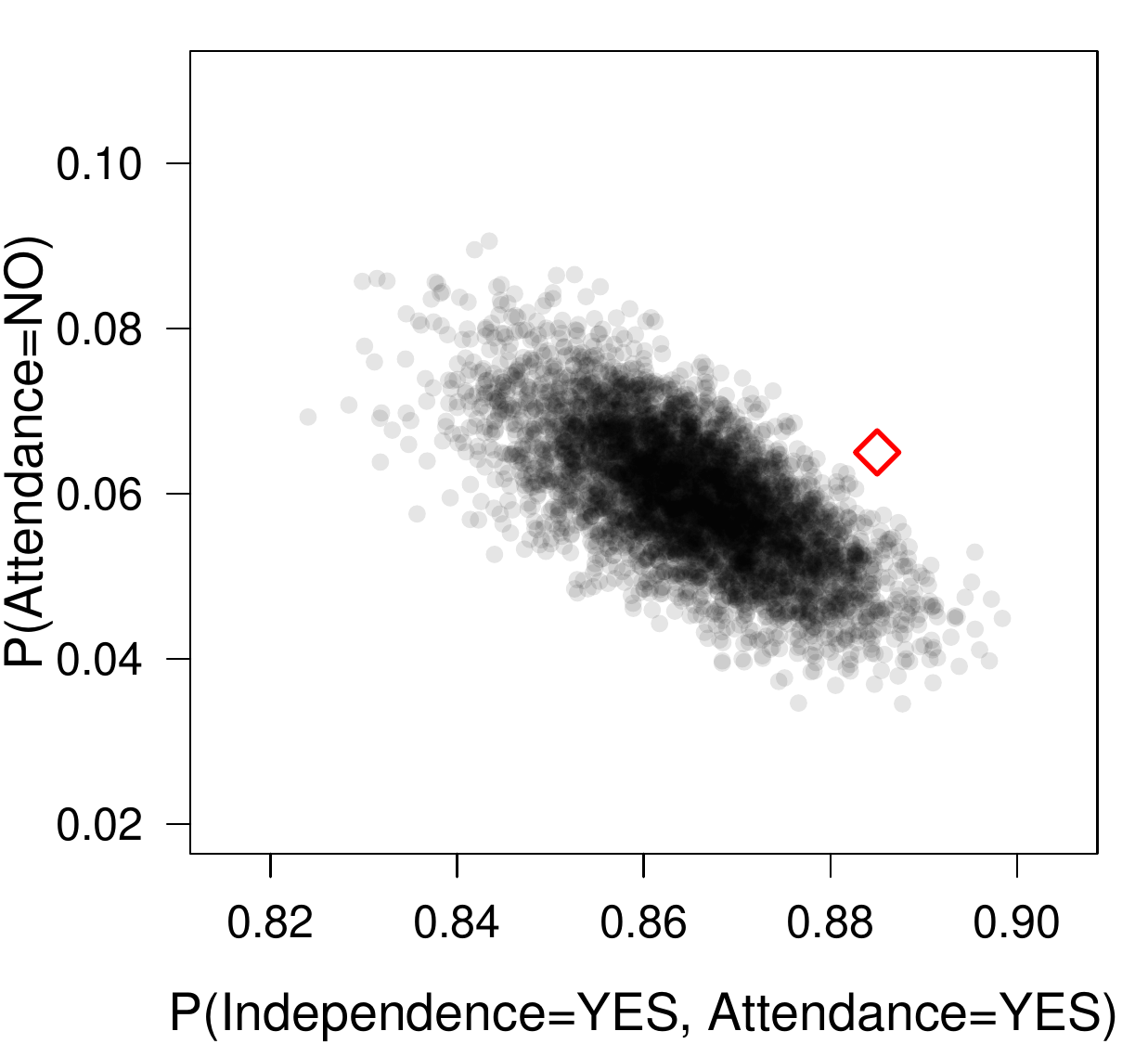}
     \end{subfigure}
		\begin{subfigure}{0.32\textwidth}
             \centering
						~~~~PIM - Secession \\\vspace{3mm}
						%\caption{}
             \label{fig:SlovPIM_Sec}\vspace{-.3cm}
             \includegraphics[width=1\textwidth]{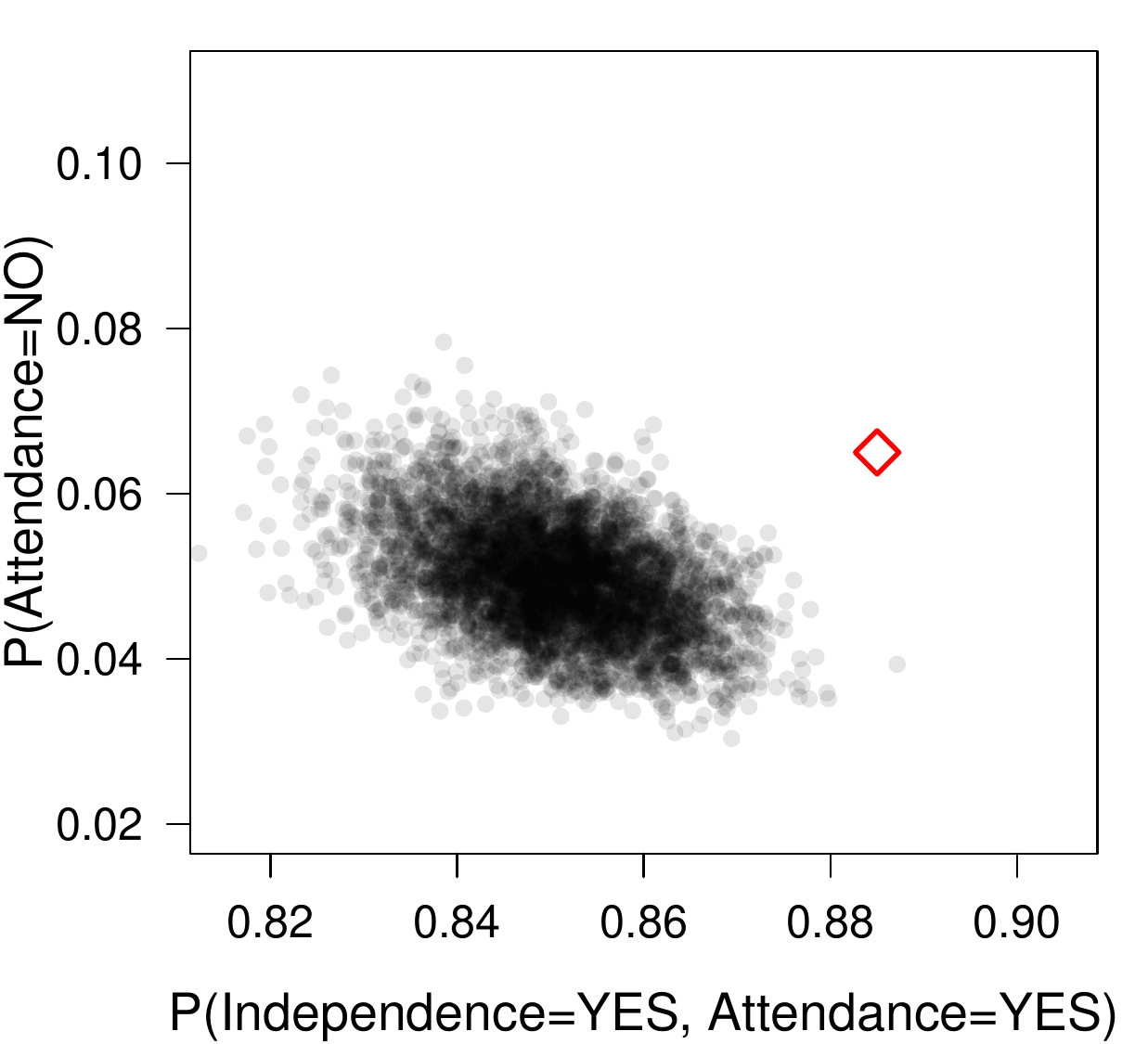}
     \end{subfigure}
\caption{Samples from joint posterior distributions of $\mathbb{P}$(Independence = Yes, Attendance = Yes) and $\mathbb{P}$(Attendance = No). Pattern mixture model under the complete-case missing-variable restriction.  The three partially ignorable missingness (PIM) mechanisms correspond to which variable we take as having ignorable missingness. The plebiscite results are represented by $\color{red}{\boldsymbol \diamond}$.
These are shown to illustrate differences between approaches and not to declare better vs worse assumptions for these data.} \label{fig:Slovenia}
\end{figure}

\section{Discussion}\label{s:disc}

The sequential identification procedure can be set up in many different ways, leading to different possibilities for constructing nonignorable missingness mechanisms.  The main differences among these possibilities lie in the assumptions about how missingness from any
one block of variables affects missingness in other blocks, as illustrated in the examples of Section \ref{ss:extreme_cases} and Section \ref{s:Examples}.
In general, the procedure  allows for different levels of dependence on missing variables while ensuring non-parametric saturated models, which 
provides a useful framework for sensitivity analysis. 

Although we presented our identification strategy for arbitrary $K$
blocks of variables, we expect that in practice most analysts would
use $K=2$ blocks when the variables do not naturally fall into ordered blocks of variables.  For example, analysts may want to partition
variables into one group that requires careful assessment of sensitivity
to various missingness mechanisms, such as outcome variables in
regression modeling with high fractions of missingness, and a second group that can be treated with generic missingness
mechanisms like conditional MAR, such as covariates with low fractions
of missingness.  These cases require partially ignorable mechanisms like those in Section \ref{s:PartialIgno}. 
Another scenario where two blocks naturally might
arise is when analysts have prior information on how the
missingness occurs for a set of variables but not for the rest.
Related, analysts might have auxiliary information 
on the marginal distribution of a few variables, perhaps from a census or other
surveys, that
enable the identification of mechanisms where the probability of nonresponse for a variable depends explicitly on the variable itself (\cite{Hiranoetal01,Dengetal13}).  

Our sequential identification procedure provides a
constructive way of obtaining estimated full-data distributions from
estimated observed-data distributions while ensuring non-parametric
saturated models.  To implement these approaches in practice, one
needs sufficient numbers of observations for each missing data
pattern, so as to allow accurate non-parametric estimation of the
observed-data distribution.   This can be challenging in modest-sized
samples with large numbers of variables.  Of course, this is the case
with most methods for handling missing data, including pattern mixture
models. In such cases, one may have to sacrifice non-parametric
saturated modeling of the observed data in favor of parametric models. 
\par

%%%%%%%%%%%%%%%%%%%%%%%%%%%%%%%%%%%%%%%%%%%%%%%%%%%%%%%%%%%%%%%%%%%%%%%%%%%%%%%%%%%%%%%%%%%%%%%%%%%%%%%%%%%%%%%%%%%%%%%%%%%%
%%%%%%%%%%%%%%%%%%%%%%%%%%%%%%%%%%%%%%%%%%%%%%%%%%%%%%%%%%%%%%%%%%%%%%%%%%%%%%%%%%%%%%%%%%%%%%%%%%%%%%%%%%%%%%%%%%%%%%%%%%%%
\vskip 14pt
\section*{Acknowledgements}

This research was supported by the grant NSF SES 11-31897.
\par

%%%%%%%%%%%%%%%%%%%%%%%%%%%%%%%%%%%%%%%%%%%%%%%%%%%%%%%%%%%%%%%%%%%%%%%%%%%%%%%%%%%%%%%%%%%%%%%%%%%%%%%%%%%%%%%%%%%%%%%%%%%%

\bibhang=1.7pc
\bibsep=2pt
\fontsize{9}{14pt plus.8pt minus .6pt}\selectfont
\renewcommand\bibname

\end{document}